\documentclass{amsart}
\usepackage{amsmath}
\newtheorem{theorem}{Theorem}[section]
\newtheorem{lemma}[theorem]{Lemma}

\theoremstyle{definition}
\newtheorem{definition}[theorem]{Definition}
\newtheorem{example}[theorem]{Example}

\newcommand{\Fp}{{\mathbb F}_p}
\newcommand{\Z}{{\mathbb Z}}
\newcommand{\fF}{{\mathfrak F}}
\newcommand{\fH}{{\mathfrak H}} 
\newcommand{\fN}{{\mathfrak N}} 
\newcommand{\cser}{\mathcal{C}}
\newcommand{\sB}{\mathcal B}
\DeclareMathOperator{\Hom}{Hom} 
\DeclareMathOperator{\Ext}{Ext}
\DeclareMathOperator{\loc}{Loc}

\title[$\fF$-hypercentral and $\fF$-hypereccentric modules]{On $\fF$-hypercentral and $\fF$-hypereccentric modules for finite soluble groups}
\author{Donald W. Barnes}
\address{1 Little Wonga Rd.\\Cremorne NSW 2090\\Australia\\}
\email{donwb@iprimus.com.au}


\begin{document}

\maketitle

\begin{abstract}   I prove the group theory analogues of some Lie and Leibniz algebra results on $\fF$-hypercentral and $\fF$-hypereccentric modules.
\end{abstract}

\section{Introduction}
The theory of saturated formations and projectors for finite soluble groups was started by Gasch\"utz in \cite{Gasc}, further developed by Gasch\"utz and Lubeseder in \cite{Gasc-L} and extended by Schunck in \cite{Sch}.  This theory is set out in Doerk and Hawkes \cite{DH}.  The analogous theory for Lie algebras was developed by Barnes and Gastineau-Hills in \cite{BGH} and for Leibniz algebras by Barnes in \cite{SchunckLeib}.

If $\fF$ is a saturated formation of soluble Lie algebras and $V,W$ are $\fF$-hypercentral modules  for the soluble Lie algebra $L$, then the modules $V \otimes W$ and $\Hom(V,W)$ are $\fF$-hypercentral by Barnes \cite[Theorem 2.1]{HyperC}, while if $V$ is $\fF$ hypercentral and $W$ is $\fF$-hypereccentric, then $V \otimes W$ and $\Hom(V,W)$ are $\fF$-hypereccentric by Barnes \cite[Theorem 2.3]{hexc}.  If $L\in\fF$ and $V$ is an $L$-module, then $V$ is the direct sum of a $\fF$-hypercentral submodule $V^+$ and a $\fF$-hypereccentric submodule $V^-$ by Barnes \cite[Theorem 4.4]{HyperC}.  The group theory analogues of these theorems are easily proved if we restrict attention to modules over the field $\Fp$ of $p$ elements, (the case which arises from considering chief factors of soluble groups), but the concepts are meaningful for modules over arbitrary fields of characteristic $p$, so I prove them in this generality.

All groups considered in this paper are finite.   If $V$ is an $FG$-module, I denote the centraliser of $V$ in $G$ by $\cser_G(V)$.  In the following, $\fF$ is a saturated formation of finite soluble groups.  By Lubeseder's Theorem, (see Doerk and Hawkes \cite[Theorem  IV 4.6, p. 368]{DH}) $\fF$ is locally defined, that is, we have for each prime $p$, a (possibly empty) formation $f(p)$ and $\fF$ is the class of all groups $G$ such that, if $A/B$ is a chief factor of $G$ of $p$-power order,  $G/\cser_G(A/B) \in f(p)$.  In this case, we write $\fF = \loc(f)$ and call it the formation locally defined by $f$.  The formation function $f$ is called {\em integrated} if, for all $p$, $f(p) \subseteq \loc(f)$.  A saturated formation always has an integrated local definition.  In this paper, I will always assume that the formation function we are using is integrated.

\section{$\fF$-hypercentral and $\fF$-hypereccentric modules}
Let $G$ be a soluble group whose order $|G|$ is divisible by the prime $p$, and let $F$ be a field of characteristic $p$.   I denote by $B_1(FG)$ the principal block of irreducible $FG$-modules.  An irreducible $FG$-module $V$ is called $\fF$-central (or $(G,\fF)$-central if I need to specify the group) if $G/\cser_G(V) \in f(p)$ and $\fF$-eccentric otherwise. For the special case of $F = \Fp$, $V$ $\fF$-central is equivalent to the split extension of $V$ by $G/\cser_G(V)$ being in $\fF$ (provided that $f$ is integrated).  A module $V$ is called $\fF$-hypercentral if every composition factor of $V$ is $\fF$-central.  It is called $\fF$-hypereccentric if every composition factor is $\fF$-eccentric.

\begin{theorem}\label{tensC} Let $V, W$ be $\fF$-central $FG$-modules.  Then $V \otimes W$ and $\Hom(V,W)$ are $\fF$-hypercentral. \end{theorem}

\begin{proof}  Let $A = \cser_G(V)$ and $B = \cser_G(W)$.  Put $C = A \cap B$.  Then $V \otimes W$ and $\Hom(V,W)$ are $G/C$-modules and $G/C \in f(p)$.  It follows  easily that $V \otimes W$ and $\Hom(V,W)$ are $\fF$-hypercentral.
\end{proof}

Proving the analogue of Barnes \cite[Theorem 2.3]{hexc} requires more work.  I first prove the analogue of Barnes \cite[Theorem 4.4]{HyperC} which follows easily from the following lemma.

\begin{lemma} \label{cecsplit}  Let  $G \in \fF$. Let $A,B$ be irreducible $FG$-modules and let $V$ be an extension of $A$ by $B$.  Suppose one of $A,B$ is $\fF$-central and the other $\fF$-eccentric.  Then $V$ splits over $A$.
\end{lemma}

\begin{proof}  A module $A$ is $\fF$-central if and only if its dual $\Hom(A,F)$ is $\fF$-central.  Hence we need only consider the case in which $A$ is $\fF$-eccentric and $B = V/A$ is $\fF$-central.

First consider the case $F=\Fp$.  Consider the split extension $X$ of $V$ by $G$.  Since $V/A$ is an $\fF$-central chief factor of $X$ and $X/V \in \fF$, we have $X/A \in \fF$.  But $A$ is $\fF$-eccentric, so $X \notin \fF$.  Therefore $X$ splits over $A$ and it follows that $V$ splits over $A$.

Now let $\{\theta_i \mid i \in I\}$ be a basis (possibly infinite) of $F$ over $\Fp$.  We now consider $V$ as an $\Fp G$-module.  Take $a \in A$, $a \ne 0$.  Then $(\Fp G)a$ is a finite-dimensional submodule of $A$, so there exists a finite-dimensional  irreducible $\Fp G$-submodule $C$ of $A$.  Then $\theta_iC$ is a submodule isomorphic to $C$. It follows that $A = \oplus_{i\in I} \theta_iC = F \otimes_{\Fp} C$  and $G/\cser_G(A) = G/\cser_G(C)$.  Thus $C$ is $\fF$-eccentric.  Similarly, there exists a finite-dimensional irreducible $\Fp G$-submodule $D$ of $B$ and $B = \oplus_{i \in I} \theta_i D = F \otimes_{\Fp} D$.  Also, $D$ is $\fF$-central.  It follows that $\Ext^1(D,C) = 0$.  But
$$\Ext^1_{FG}(B,A) = F \otimes_{\Fp} \Ext^1_{\Fp G}(D,C) =0.$$
Therefore $V$ splits over $A$.
\end{proof}

\begin{theorem} \label{components}    Let $G \in \fF$ and let $V$ be an $FG$-module.  Then there exists a direct decomposition $V = V^+ \oplus V^-$ where $V^+$ is $\fF$-hypercentral and $V^-$ is $\fF$-hypereccentric.
\end{theorem}

\begin{proof}  Let $V_0 = 0 \subset V_1 \subset \dots \subset V_n = V$ be a composition series of $V$.  If for some $i$, we have $V_i/V_{i-1}$ $\fF$-eccentric and $V_{i+1}/V_i$ $\fF$-central, by Lemma \ref{cecsplit}, we can replace $V_i$ by a submodule $V'_i$ between $V_{i+1}$ and $V_{i-1}$, so bringing the $\fF$-central factor below the $\fF$-eccentric factor.  By repeating this, we obtain a composition series in which all $\fF$-central factors are below all $\fF$-eccentric factors.  This gives us an $\fF$-hypercentral submodule $V^+$ with $V/V^+$ $\fF$-hypereccentric.  Likewise, we can bring all $\fF$-eccentric factors below the $\fF$-central factors,  so obtaining an $\fF$-hypereccentric submodule $V^-$ with $V/V^-$$\fF$-hypercentral.  Clearly, $V = V^+ \oplus V^-$.
\end{proof}

For Lie and Leibniz algebras, there is a strengthened form (\cite[Lemma 1.1]{extras} and \cite[Theorem 3.19]{SchunckLeib}) of this theorem.  If $U$ is a subnormal subalgebra of the not necessarily soluble algebra $L$ and $V$ is an $L$-module, then the $U$-module components $V^+, V^-$ are $L$-submodules.

\begin{theorem} \label{normcomp}    Let $U \in \fF$ be a normal subgroup of the not necessarily soluble group $G$ and let $V$ be an $FG$-module.  Then the $(U, \fF)$-components $V^+, V^-$ are $FG$-submodules.
\end{theorem}

\begin{proof} Let $W$ be either of $V^+,V^-$and let $g \in G$.  Consider the action of $u \in U$ on $gW$.  We have $ugW = g(g^{-1}ug)W \subseteq gW$.  Thus $gW$ is a $U$-submodule of $V$.  If $A$ is a composition factor of $W$, then $gA$ is a composition factor of $gW$, and $\cser_U(gA) = g \cser_U(A)g^{-1}$.  Thus $U/\cser_U(gA) \simeq U/\cser_U(A)$ and $gA$ is $(U,\fF)$-central if and only if $A$ is $(U,\fF)$-central.  Thus $gV^+$ is $(U, \fF)$-hypercentral and so $gV^+ \subseteq V^+$ for all $g \in G$.  Similarly, $gV^- \subseteq V^-$ for all $g \in G$.
\end{proof}

\begin{example} \label{nonsubn} Let $G$ be the group of permutations of the set of symbols $\{e_1, \dots, e_6\}$ generated by the permutations of $\{e_1, e_2, e_3\}$ and the permutation $(14)(25)(36)$ and let $U$ be the subgroup generated by the permutation $(123)$. Then $U$ is subnormal in $G$, being normal in the subgroup $N$ consisting of those permutations which map $\{e_1, e_2, e_3\}$ into itself.  Let $F$ be a field of characteristic $2$ and let $V$ be the vector space over $F$ with basis $\{e_1, \dots, e_6\}$.  Let $\fF$ be the saturated formation of all nilpotent groups.  Then $\fF$ is locally defined by the function $f(p)= \{1\}$ for all primes $p$.  Considering $V$ as $U$-module, we have $V^+ = \langle e_1+e_2+e_3, e_4, e_5, e_6 \rangle$ and $V^- = \langle e_1+  e_2,  e_2+  e_3\rangle$.  (If $F$ contains a root of $x^2+x+1$, then $V^-$ is the direct sum of two non-trivial 1-dimensional modules.  Otherwise, $V^-$ is irreducible.)   $V^+$ and $V^-$ are invariant under the action of $N$ but not invariant under the action of $G$.  Thus Theorem \ref{normcomp} cannot be extended to $U$ subnormal.
\end{example}

Suppose $G \in \fF$.  Clearly, the trivial $FG$-module $F$ is $\fF$-central.  From Theorem \ref{components}, it follows that if $V$ is an irreducible $FG$-module in the principal block, then $V$ is $\fF$-central.   Since $H^n(G,V)=0$ for all $n$ if $V$ is not in the principal block, it follows that for any $\fF$-hypereccentric module $V$, we have $H^n(G,V)=0$ for all $n$.  We cannot conclude from $H^n(G,V)=0$ for all $n$ that $V$ is $\fF$-hypereccentric as, for any $V$, there is some $\fF$ for which $V$ is $\fF$-hypercentral.  To obtain a sufficient condition for $V$ to be $\fF$-hypereccentric, we use the $\fF$-cone over $G$.

\begin{definition}  Suppose $G \in \fF$.  The $\fF$-cone over $G$ is the class $(\fF/G)$ of all pairs $(X, \epsilon)$ where $X \in \fF$ and $\epsilon: X \to G$  is an epimorphism.  We usually omit $\epsilon$ from the notation, writing simply $X \in (\fF/G)$.
\end{definition}

Any $FG$-module $V$ is an $FX$-module via $\epsilon$ for any $X \in (\fF/G)$.  Then $V$ is $\fF$-hypercentral or $\fF$-hypereccentric as $FX$-module if and only if it is $\fF$-hypercentral, respectively $\fF$-hypereccentric as $FG$-module.

\begin{theorem} \label{hec-crit}   Suppose $G \in \fF$ and that $H^1(X,V)=0$ for all $X \in (\fF/G)$.  Then $V$ is $\fF$-hypereccentric.
\end{theorem}

\begin{proof} By Theorem \ref{components}, $V$ is a direct sum of a $\fF$-hypercentral module and a $\fF$-hypereccentric module.  Thus, without loss of generality, we may suppose that $V$ is $\fF$-hypercentral and we then have to prove that $V=0$.   So suppose that $V \ne 0$.  There exists a minimal $\Fp G$-module $W$ of $V$. ($W$ is finite-dimensional, whatever the field $F$.)  We form the direct sum $A$ of sufficiently many copies of $W$ to ensure that $\dim_F \Hom_{\Fp G}(A,V) > \dim H^2(G,V)$.  Let $X$ be the split extension of $A$ by $G$.  As $W$ is $\fF$-central, $X \in (\fF/G)$ and by assumption, $H^1(X,V)=0$.  We use the Hochschild-Serre spectral sequence to calculate $H^1(X,V)$.  We have 
$$E^{2,0}_2 = H^2(X/A, V^A) = H^2(G,V)$$
and
$$E^{0,1}_2 = H^0(X/A, H^1(A,V)) = \Hom_{\Z}(A,V)^G = \Hom_{\Fp G}(A,V).$$
Now $d^{0,1}_2$ maps $E^{0,1}_2$ into $E^{2,0}_2 = H^2(G,V)$.  As $\dim H^2(G,V) < \dim E^{0,1}_2$, we have $\ker(d^{0,1}_2) \ne 0$.  So $E^{0,1}_3 \ne 0$ and $H^1(X,V) \ne 0$ contrary to assumption.
\end{proof}

\begin{theorem} \label{cent-ecc}   Suppose $G \in \fF$.  Suppose that $V$ is an $\fF$-hypercentral $FG$-module and that $W$ is an $\fF$-hypereccentric $FG$-module.  Then $V \otimes W$ and $\Hom(V,W)$ are $\fF$-hypereccentric.
\end{theorem}

\begin{proof}  Let $X \in (\fF/G)$.  Then $V$ and $W$ are $\fF$-hypercentral and $\fF$-hypereccentric respectively as $FX$-modules.  Every $FX$-module extension of $W$ by $V$ splits.  Thus $H^1(X,\Hom(V,W)) = 0$.  By Theorem \ref{hec-crit}, $\Hom(V,W)$ is $\fF$-hypereccentric.  The dual module $V^* = \Hom(V,F)$ is also $\fF$-hypercentral. As
$$V \otimes W \simeq V^{**} \otimes W \simeq \Hom(V^*,W),$$
we have also that $V \otimes W$ is $\fF$-hypereccentric.
\end{proof}

\section{Blocks}
Let $F$ be a field of characteristic $p$ and let $\fF$ be a saturated formation of finite soluble groups locally defined by the function $f$ with $f(p) \ne \emptyset$.  Let $U \in \fF$ be a normal subgroup of the not necessarily soluble group $G$.  Then the direct decomposition $V = V^+ \oplus V^-$ with respect to $U$ given by Theorem \ref{normcomp} is natural.  But if we take a partition $\sB = \sB^+ \cup \sB^-$ of the set $\sB$ of blocks of $FG$-modules, we have a natural direct decomposition $V = V^+ \oplus V^-$ of $FG$-modules where every irreducible composition factor of $V^+$ is in a block in $\sB^+$ and every composition factor of $V^-$ is in a block in $\sB^-$.  Further, every natural direct decomposition of $FG$-modules has this form.  Thus the $(U,\fF)$ direct decomposition $V = V^+ \oplus V^-$ is the $(\sB^+, \sB^-)$ decomposition for some partition of $\sB$.  It follows that if some irreducible $FU$-module in an $FG$-block $B$ is $\fF$-central, then all irreducibles in $B$ are $\fF$-central.  The special case of this where $U=G$ and $F = \Fp$ has been  proved without assuming $\fF$ locally defined  as a stage in a proof that all saturated formations are locally definable.  (See Doerk and Hawkes \cite[Lemma IV 4.4]{DH}.)  I investigate the relationship between $\fF$ and the partition $(\sB^+, \sB^-)$.

\begin{lemma} \label{chiefplus}  Let $A/B$ be a $p$-chief factor of $U$ and let $V = F \otimes_{\Fp}(A/B)$.  Then $V$ is $(U,\fF)$-hypercentral.
\end{lemma}

\begin{proof} Since $U \in \fF$, $U/\cser_U(A/B) \in f(p)$.  Therefore $U/\cser_U(V) \in f(p)$.
\end{proof}

Green and Hill have proved (see Doerk and Hawkes \cite[Theorem B 6.17, p.136]{DH}  that if $A/B$ is a $p$-chief factor of a $p$-soluble group $U$, then $A/B \in B_1(\Fp U)$.  The following lemma generalises this.

\begin{lemma}  \label{chiefblock} Let $A/B$ be a $p$-chief factor of the $p$-soluble group $U$. Then every composition factor of $F \otimes_{\Fp}(A/B)$ is in $B_1(FU)$.
\end{lemma}

\begin{proof} The largest $p$-nilpotent normal subgroup $O_{p'p}(U)$ is the intersection of the centralisers of the $p$-chief factors of $U$, so $O_{p'p}(U) \subseteq \cser_U(A/B)$. If $V$ is a composition factor of $F \otimes_{\Fp}(A/B)$, then $\cser_U(V) \supseteq \cser_U(A/B)$.  By a theorem of Fong and Gasch\"utz, (see Doerk and Hawkes \cite[Theorem B 4.22, p. 118]{DH}) an irreducible $FU$-module $V$ is in the principal block if and only if $O_{p'p}(U) \subseteq \cser_U(V)$.   Therefore $V \in B_1(FU)$.
\end{proof}

\begin{lemma} \label{princ} Let $V$ be an irreducible $FG$-module in the principal block.  Then $V$ is $(U,\fF)$-hypercentral.
\end{lemma}

\begin{proof}  Since $V \in B_1(G)$, there exists a chain of irreducible $G$-modules $V_0,  \dots, V_n = V$ where $V_0$ is the trivial module, and modules $X_1, \dots, X_n$ where $X_i$ is a non-split extension of one of $V_{i-1}, V_i$ by the other.  If $V$ is not $\fF$-central, then for some $k$, we have $V_{k-1}$ $\fF$-central and $V_k$ 
$\fF$-eccentric.  But then, by Theorem \ref{normcomp}, $X_i$ cannot be indecomposable.
\end{proof}

If $\fF$ is the smallest saturated formation containing the soluble group $G$ and $N = O_{p'p}$ is the largest $p$-nilpotent normal subgroup of $G$, then for $f(p)$, we may take the smallest formation containing $G/N$.

\begin{lemma} \label{fpres} Let $K$ be the $f(p)$-residual of $G$.  Then $K = N$.
\end{lemma}

\begin{proof}  Let $G$ be a minimal counterexample.  Let $A$ be a minimal normal subgroup of $G$.  Suppose $A \subseteq O_{p'}$.  Then $N/A = O_{p'p}(G/A)$ and the result holds for $G/A$ and so also for $G$.  It follows that $O_{p'}(G) = \{1\}$ and $A \subseteq N$.  If $A \ne N$, again we have that the result holds for $G/A$ and for $G$.  Therefore $N$ is the only minimal normal subgroup of $G$.

The class $\fN^r$ of groups of nilpotent length $r$ is a formation.  If $G$ has nilpotent length $r$, then $G/N$ has nilpotent length $r-1$ and it follows that $f(p) \subseteq \fN^{r-1}$.  Thus the $f(p)$-residual of $G$ cannot be $\{1\}$ and so must be $N$, contrary to the assumption that $G$ is a counterexample.
\end{proof}

\begin{theorem} \label{onlyB1}  Let $G$ be a finite soluble group and let $F$ be a field of characteristic $p$.  Let $\fF$ be the smallest saturated formation containing $G$ and let $V$ be an irreducible $FG$-module.  Then $V$ is $\fF$-central if and only if $V \in B_1(FG)$.
\end{theorem}

\begin{proof}  By Lemma \ref{princ}, if $V \in B_1(FG)$, then $V$ is $\fF$-central.  Conversely, if $V$ is $\fF$-central, then $\cser_G(V) \supseteq O_{p'p}(G)$ and $V \in B_1(FG)$ by the Fong-Gasch\"utz Theorem \cite[B 4.22, p.118]{DH}.
\end{proof}

Theorem \ref{onlyB1} does not need the full force of the assumption that $\fF$ is the smallest saturated formation containing $G$, merely that $f(p)$ is minimal.  Restrictions on the $f(q)$ for $q \ne p$ are irrelevant.

\bibliographystyle{amsplain}

\begin{thebibliography}{10}


\bibitem{HyperC} D. W. Barnes, \textit{On $\mathfrak F$-hypercentral modules for  Lie algebras}, Archiv der Math. \textbf{30} (1978), 1--7.

\bibitem{hexc} D. W. Barnes, \textit{On $\fF$-hyperexcentric modules for
Lie algebras},  J. Austral. Math. Soc. \textbf{74} (2003), 235--238.

\bibitem{extras} D. W. Barnes, \textit{Ado-Iwasawa extras}, J. Austral. Math. Soc.
\textbf{78} (2005), 407--421. 


\bibitem{SchunckLeib} D. W. Barnes, \textit{Schunck classes of soluble Leibniz algebras,} arXiv:1101.3046 (2011).

\bibitem{BGH} D. W. Barnes and H. M. Gastineau-Hills,  \textit{On the theory of soluble Lie
algebras,}  Math. Zeitschr.  \textbf{106} (1968), 343--354.

\bibitem{CR} C. W. Curtis and I. Reiner, \textit{Representation theory of finite groups and associative algebras,} Interscience, New York 1962.

\bibitem{DH} K. Doerk and T. Hawkes, \textit{Finite soluble groups,} DeGruyter,  Berlin-New York
1992.


\bibitem{Gasc} W. Gasch\"utz, \textit{Zur Theorie der endlichen aufl\"osbaren gruppen,} Math. Zeitschr. \textbf{80} (1963), 300--305.

\bibitem{Gasc-L} W. Gasch\"utz and U. Lubeseder, \textit{Kennzeichnung ges\"attigter Formationen,} Math. Zeitschr. \textbf{82} (1963), 198--199.

\bibitem{Sch} H. Schunck, \textit{$\fH$-Untergruppen in endlichen aufl\"osbaren Gruppen,} Math. Zeitschr. \textbf{97} (1967), 326--330.



\end{thebibliography}

\end{document}